\documentclass[10pt,oneside,a4paper,draft]{article} 	% \documentclass[12pt,twosides,draft]{article}
\usepackage[fleqn]{amsmath} 						% to move the displaied equations to the left
\usepackage{amsthm,amsfonts,latexsym,amssymb,amscd} 	% AMS macros 
\usepackage[mathscr]{eucal}   						% to use the font eucal (\As ... \Zs below)  
\usepackage[all]{xy}									% for diagrams
\usepackage[draft=false]{hyperref} 								% for pdf links   
\newcommand{\F}{\mathcal{F}}
\renewcommand{\H}{\mathcal{H}} 	%re!
\renewcommand{\O}{\mathcal{O}}	%re! 
\newcommand{\Fg}{\mathfrak{F}}

\newcommand{\Sg}{\mathfrak{S}}
\newcommand{\Tg}{\mathfrak{T}}

\newcommand{\CC}{{\mathbb{C}}}

\newcommand{\RR}{{\mathbb{R}}} 
\newcommand{\As}{{\mathscr{A}}}\newcommand{\Bs}{{\mathscr{B}}}\newcommand{\Cs}{{\mathscr{C}}}

\newcommand{\Es}{{\mathscr{E}}}\newcommand{\Fs}{{\mathscr{F}}}\newcommand{\Gs}{{\mathscr{G}}}

\newcommand{\Ms}{{\mathscr{M}}}\newcommand{\Ns}{{\mathscr{N}}}\newcommand{\Os}{{\mathscr{O}}}
\newcommand{\Rs}{{\mathscr{R}}}
\newcommand{\Xs}{{\mathscr{X}}}
\newcommand{\Ys}{{\mathscr{Y}}}\newcommand{\Zs}{{\mathscr{Z}}}
 
\DeclareFontFamily{U}{rsfs}{\skewchar\font127 }
\DeclareFontShape{U}{rsfs}{m}{n}{% 
   <5> <6> rsfs5
   <7> rsfs7
   <8> <9> <10> <10.95> <12> <14.4> <17.28> <20.74> <24.88> rsfs10
}{}
\DeclareSymbolFont{rsfs}{U}{rsfs}{m}{n} 
\DeclareSymbolFontAlphabet{\scr}{rsfs}% was ... mathscr
 
\newcommand{\Af}{\scr{A}}\newcommand{\Ef}{\scr{E}}
\newcommand{\Ff}{\scr{F}}\newcommand{\Hf}{\scr{H}}
\newcommand{\Kf}{\scr{K}}\newcommand{\Lf}{\scr{L}}\newcommand{\Mf}{\scr{M}}\newcommand{\Of}{\scr{O}} 
\newcommand{\Pf}{\scr{P}}\newcommand{\Rf}{\scr{R}}\newcommand{\Sf}{\scr{S}}\newcommand{\Tf}{\scr{T}}
\newcommand{\Vf}{\scr{V}}

\DeclareMathOperator{\Ob}{Ob} 

\DeclareMathOperator{\Hom}{Hom}

\renewcommand{\emph}{\textbf} 										% emphasis with boldface
\newcommand{\cs}{C*}

\newcommand{\hlink}[2]{\href{#1}{\texttt{#2}}} % usage: \hlink{URL}{TEXT} to make hyperlink with text anchor not necessarily given their URL in typewriter stile. 

\newtheorem{theorem}{Theorem}[section]							% theorems numbered under sections 

\newtheorem{definition}[theorem]{Definition} 
\newtheorem{remark}[theorem]{Remark}
\newtheorem{example}[theorem]{Example}
\numberwithin{equation}{section}  		% {chapter} to enumerate equations after sections
\title{\textbf{Enriched Fell Bundles and Spaceoids}}

\author{\normalsize Paolo Bertozzini$^a$, Roberto Conti$^b$, Wicharn Lewkeeratiyutkul$^c$ 
\\
\normalsize 
$^a$\textit{Department of Mathematics and Statistics, Faculty of Science and Technology,} 
\\
\normalsize 
\textit{Thammasat University, Pathumthani 12121, Thailand,} 
\texttt{paolo.th@gmail.com} 
\\
\normalsize 
$^b$\textit{Dipartimento di Scienze, Universit\`a di Chieti-Pescara ``G. D'Annunzio'',} 
\\
\normalsize \textit{Viale Pindaro 42, I-65127 Pescara, Italy,} 
\texttt{conti@sci.unich.it}  
\\
\normalsize 
$^c$\textit{Department of Mathematics and Computer Science, Faculty of Science,} 
\\ 
\normalsize 
\textit{Chulalongkorn University, Bangkok 10330, Thailand,} 
\texttt{Wicharn.L@chula.ac.th}
}

\date{\normalsize{30 November 2009, revised: %06 Semptember 2011, 
26 November 2011}}
 
\begin{document}

\maketitle

\begin{abstract} \noindent  
We propose a definition of involutive categorical bundle (Fell bundle) enriched in an involutive monoidal category and we
argue that such a structure is a possible suitable environment for the formalization of different equivalent versions of spectral data for commutative C*-categories. 

\medskip 
 
\noindent 
\textbf{MSC-2010:} 
					46L87,			% Non-commutative Geometry 
					46M15, 			% Categories and Functors (in Functional Analysis)  
					46L08,			% C*-modules 	  				 
					46M20, 			% Methods of algebraic topology (cohomology, sheaf and bundle theory, etc.)
					16D90,			% Module Categories, module theory in a category-theoretic context, Morita equivalence and duality
					18F99. 			% Categories and Geometry

\medskip
 
\noindent
\textbf{Keywords:} Involutive Monoidal Category, Enriched Category, Fell Bundle, Spaceoid, 
\\ 
Non-commutative Geometry. 
\end{abstract}

\section{Introduction}

In a previous work~\cite{BCL4} we obtained a horizontal categorification\footnote{A terminology that we introduced 
in~\cite[Section~4.2]{BCL0}.} of Gel'fand duality theorem for commutative unital 
C*-algebras that provides a spectral theory for commutative full \hbox{C*-categories} in terms of certain Hermitian line bundes (over the product of a compact Hausdorff space with a connected equivalence relation) naturally equipped with a Fell bundle structure, that we called ``spaceoids''.  

Modulo a suitable definition of ``Fell bundle enriched in a monoidal $*$-category'' (or more generally an involutive bicategory), there are actually three completely equivalent ways to see the spectrum of a commutative full \hbox{C*-category}:   
\begin{itemize}
\item[a)] 
as a ``continuous field of one dimensional C*-categories'' (i.e.~a Fell bundle, over the ``diagonal equivalence relation'' 
$\Delta_X$ of a compact Hausdorff space $X$, ``enriched'' in the \hbox{$*$-bicategory} 
of ``Hilbert bimodules'' of one-dimensional 
C*-categories with a fixed set of objects $\Os$ (with ``monoidal'' product given by the tensor product of such bimodules and ``monoidal'' involution given by the Rieffel dual of the bimodules), 
\item[b)] 
a Fell bundle, over the maximal equivalence relation $\Rs_\Os=\Os\times\Os$ of a discrete space $\Os$, ``enriched'' in the monoidal 
$*$-category of Hermitian line bundles over a fixed compact Hausdorff space $X$ (with monoidal product given by the fiberwise tensor product and monoidal  involution given by the fiberwise dual), 
\item[c)]
as a compact topological spaceoid i.e.~a rank-one unital Fell bundle over the product equivalence relation $\Delta_X\times\Rs_\O$ of a compact Hausdorff space $X$ and a discrete space $\Os$, as described in our previous paper. 
\end{itemize} 

Each of the above pictures has its own advantages: ``equivalence relation of Hermitan line-bundles'' being the easiest picture of the spectrum just as a ``strictification'' of an equivalence relation in the geometric Picard-Rieffel group of $X$; 
``continuous field of one-dimensional \hbox{C*-categories}'' being the most natural in terms of the construction of the spectrum; the description via ``spaceoids'' being the most convenient in terms of Fell bundles (since, putting on the same foot discrete and continuous variables, it does not require any ``iteration'' of the Fell bundle construction) and the one that we deem more ``appropriate'' for further generalizations. 

It is the purpose of our present work to propose the suitable definitions of ``Fell bundles enriched in a monoidal involutive category'' and to prove rigorously the previous equivalences between different descriptions of spectral data for commutative full C*-categories. 

\medskip 

The content of the paper is as follows. 

In section~\ref{sec: preliminaries} we recall the basic definitions of \hbox{C*-categories} and Fell bundles and we introduce 
a definition of spaceoid that 
is suitable for 
the statement of our Ge'fand duality theory for small commutative full C*-categories.  
The main part of the paper is contained in section~\ref{sec: monoidal} where we propose a ``bundle version'' of M.~Kelly enriched categories and we further consider categorical bundles over a $*$-category that are enriched in a monoidal $*$-category (or a \hbox{$*$-bicategory}). 
In section~\ref{sec: equivalence} we finally show how categorical bundles enriched in monoidal \hbox{$*$-categories} (\hbox{$*$-bicategories}) can be used to formalize different but perfectly equivalent versions of the spectral data for a commutative full C*-category. 
We conclude with a short outlook (section~\ref{sec: outlook}) on possible future developments and applications. 

\section{Preliminaries on Fell Bundles and Spaceoids}\label{sec: preliminaries}

For the convenience of the reader, we provide here the basic definitions of C*-categories, Fell bundles and spaceoids and recall our Gel'fand duality result between small commutative full C*-categories and compact Hausdorff spaceoids. 

\begin{definition}
A \emph{$*$-category} (also called \emph{dagger category} or \emph{involutive category})\footnote{The definition from M.~Burgin~\cite{Bu}, essentially implicit in~\cite{GLR}, has been (re-)introduced and used by several authors, see 
e.g.~J.~Lambek~\cite{La}, P.~Mitchener~\cite{M1}, S.~Abramski-B.~Coecke~\cite{AC}, P.~Selinger~\cite{S}.} is a category 
$\Cs$ with contravariant functor $*:\Cs\to\Cs$ acting identically on the objects that is involutive i.e.~$(x^*)^*=x$ for all 
$x\in\Hom_ \Cs$. 
A $*$-category is an \emph{inverse $*$-category} if $xx^*x=x$ for all $x\in \Hom_\Cs$. Whenever $\Hom_\Cs$ is equipped with a topology, we require the compositions and involutions to be continuous. 
\end{definition}
For a ($*$-)category $\Cs$ we will denote by $\Cs^o$ the set of identities $\iota_A$, for $A\in \Ob_\Cs$ and by $\Cs^n$ the set of 
$n$-tuples $(x_1,\dots,x_n)$ of composable arrows $x_1,\dots,x_n\in \Hom_\Cs$.

\begin{definition} 
A \emph{C*-category}\footnote{First introduced in P.~Ghez-R.~Lima-J.~Roberts~\cite{GLR} and further developed in P.~Mitchener~\cite{M1}.} is a $*$-category $\Cs$ such that: the sets $\Cs_{AB}:=\Hom_\Cs(B,A)$ are complex Banach spaces; the compositions are bilinear maps such that $\|xy\|\leq\|x\|\cdot \|y\|$, for all $x\in \Cs_{AB}$, $y\in \Cs_{BC}$; the involutions are conjugate linear maps 
such that $\|x^*x\|=\|x\|^2$, $\forall x\in \Cs_{BA}$ and such that $x^*x$ is a positive element in the C*-algebra $\Cs_{AA}$, for every $x\in \Cs_{BA}$ (i.e.~$x^* x = y^* y$ for some $y \in \Cs_{AA}$). A C*-category is \emph{commutative} is for all 
$A\in \Ob_\Cs$, the C*-algebras $\Cs_{AA}$ are commutative. 
\end{definition}

Following J.~Fell-R.~Doran~\cite[Section~I.13]{FD} or N.~Waever~\cite[Chapter~9.1]{W} we have the following definition of Banach 
bundle\footnote{Note that our norms are supposed to be continuous.}. 
\begin{definition}\label{def: banach} 
A \emph{Banach bundle} $(\Es,\pi,\Xs)$ is a surjective continuous open map $\pi:\Es\to\Xs$ such that for all $x\in \Xs$ the fiber 
$\Es_x:=\pi^{-1}(x)$ is a complex Banach space and satisfying the following additional conditions:
\begin{itemize}
\item
the operation of additon $+:\Es\times_\Xs\Es\to\Es$ is continuous on the set \\ 
$\Es\times_\Xs\Es:=\{(x,y)\in \Es\times\Es \ | \ \pi(x)=\pi(y)\}$, 
\item
the operation of multiplication by scalars $\cdot:\CC\times\Es\to \Es$ is continuous,  
\item
the norm $\|\cdot\|:\Es\to\RR$ is continuous, 
\item
for all $x_o\in \Xs$, the family $U_{x_o}^{\Os,\epsilon}=\{e\in \Es \ | \ \|e\|<\epsilon, \pi(e)\in \Os\}$ where $\Os\subset \Xs$ is an open set containing $x_o\in\Xs$ and $\epsilon>0$, is a fundamental system of neighbourhood of $0\in E_{x_o}$. 
\end{itemize}
For a \emph{Hilbert bundle} we require that for all $x\in \Xs$, fiber $\Es_x$ is a Hilbert space. 
\end{definition}

\begin{definition}\label{def: fell-bundle} 
A \emph{Fell bundle}\footnote{Fell bundles over topological groups were first introduced by J.~Fell~\cite[Section~II.16]{FD} and later generalized to the case of groupoids by S.~Yamagami (see A.~Kumjian~\cite{K} and references therein) and to the case of inverse semigroups by N.~Sieben (see R.~Exel~\cite[Section~2]{E}).} $(\Es,\pi,\Xs)$ over an inverse $*$-category $\Xs$ is a Banach bundle that is also a $*$-category $\Es$ fibered over the $*$-category $\Xs$ with continuous fiberwise bilinear compositions and fiberwise conjugate-linear involutions such that \hbox{$\|ef\|\leq \|e\|\cdot\|f\|$} for all composable $e,f\in \Es$, $\|e^*e\|=\|e\|^2$ for all $e\in \Es$ and $e^*e$ is a positive element in the \hbox{C*-algebra} 
$\Es_{\pi(e^*e)}:=\{f\ | \ \pi(f)=\pi(e^*e)\}$. 

By a \emph{morphism of Fell bundles} $(\Es_j,\pi_j,\Xs_j)$, for $j=1,2$, we mean a pair $(\phi,\Phi)$ such that $\phi:\Xs_1\to\Xs_2$ is a $*$-isomorphism of $*$-categories and $\Phi:\Es_1\to\Es_2$ is a continuous fiberwise linear $*$-functor such that $\phi\circ\pi_1=\pi_2\circ\Phi$. 
\end{definition}

\begin{remark}\label{rem: banach}
For essential future use in section~\ref{sec: monoidal}, note that the axioms defining a Banach, Hilbert or Fell bundle can be imposed also to situations in which the fibers $\Es_x$ are themselves disjoint union of Banach/Hilbert spaces and in particular when, for all 
$x\in \Xs$, $\Es_x$ is actually the total space of another Banach/Hilbert/Fell bundle over another space $\Ys$. 
\end{remark}

\begin{remark} 
In a C*-category $\Cs$, the set $\Ob_\Cs\times\Ob_\Cs$ is a groupoid (actually a maximal equivalence relation) and so small C*-categories can be identified as Fell bundles over such an equivalence relation. 

Note also that in a Fell bundle $(\Es,\pi,\Xs)$, whenever $x^*=x^*x=x\in \Xs$ is an Hermitian idempotent, the fibers $\Es_x:=\pi^{-1}(x)$ are 
C*-algebras and that for all $x\in \Xs$, the fiber $\Es_x$ is a Hilbert C*-bimodule over the C*-algebras $\Es_{xx^*}$ on the left and 
$\Es_{x^*x}$ on the right.  
\end{remark}
\begin{definition}
A Fell bundle is \emph{unital} if, for all Hermitian idempotents $x\in \Xs$, the \hbox{C*-algebras} $\Es_x$ are unital. 
A Fell bundle is \emph{commutative} if the category $\Es$ is commutative (equivalently, if the monoid $\pi^{-1}(\Xs_{AA})$ is  commutative for all 
$A\in \Ob_\Xs$).\footnote{This implies that the inverse $*$-category $\Xs$ is commutative (i.e.~$\Xs_{AA}$ is a commutative monoid for all 
\hbox{$A\in \Ob_\Xs$)} and for all Hermitian idempotents $x\in \Xs$ the C*-algebra $\Es_x$ is commutative.} 

A Fell bundle is \emph{saturated} if all the Hilbert C*-bimodules $\Es_x$ are full. 
A small C*-category is a \emph{full C*-category} if it is saturated as a Fell bundle. A small C*-category is unital (respectively commutative) whenever it is unital (respectively commutative) as a Fell bundle. 
\end{definition}

In following we will restrict our attention only to the simple case of unital saturated commutative Fell bundles whose base category 
$\Xs$ is actually an equivalence relation. 
Clearly the category $\Af$ of object-bijective $*$-functors between small commutative full \hbox{C*-categories} considered 
in~\cite{BCL4} is a full subcategory of the category $\Ef$ of morphisms of commutative saturated unital Fell bundles.  

\begin{definition}
A \emph{(compact Hausdorff) topological spaceoid for a commutative full C*-category} (or simply a spaceoid, for short) $(\Es,\pi,\Xs)$ is a unital rank-one Fell bundle over the product involutive topological category $\Xs:=\Delta_X\times\Gs_\Os$ where 
$\Delta_X:=\{(p,p) \ | \ p\in X\}$ is the minimal equivalence relation of a compact Hausdorff space $X$ and $\Gs_\Os:=\Os\times\Os$ is a the total equivalence relation on a discrete space $\Os$.
A \emph{morphism of spaceoids} $(\Es_j,\pi_j,\Xs_j)$, for $j=1,2$, is given by a pair $(f,\F)$, where $f=(f_\Delta,f_\Gs)$ with 
$f_\Delta:\Delta_{X_1}\to \Delta_{X_2}$ continuous map between compact Hausdorff spaces, \hbox{$f_\Gs:\Gs_1\to\Gs_2$} isomorphism of groupoids and where $\F:f^\bullet(\Es_2)\to \Es_1$ is continuous \hbox{$*$-functor} that is also a fiberwise linear bundle map over $\Xs_1$.\footnote{Recall that with the same notation used in~\cite[Section~3]{BCL4}, $f^\bullet(\Es_2)$ denotes the total space of the standard pull-back $(f^\bullet(\Es_2),\pi^f,\Xs_1)$ of the Fell bundle $\Es_2$ under the map $f:\Xs_1\to\Xs_2$ and that \hbox{$f^{\pi_2}:f^\bullet(\Es_2)\to \Es_2$} is the unique bundle morphism such that $f\circ\pi_2^f=\pi_2\circ f^{\pi_2}$.} 
\end{definition} 
Note that for all $p\in \Delta_X$, $\Es_p:=\uplus_{g\in \Gs}\Es_{p_g}=\pi^{-1}(\{p\}\times\Gs_\Os)$ is basically a rank-one Fell bundle over $\Gs_\Os$. Furthermore, it is always possible to find a $*$-functor $\gamma_p:\Es_p\to\CC$. 
Note also that, since a constant finite rank Fell bundle is a Hermitian vector bundle, $\Es_g:=\pi^{-1}(\Delta_X\times\{g\})$ is actually a Hermitian vector bundle over $\Delta_X$. 
\begin{remark}
Morphisms of spaceoids constitute a category $\Tf_\Af$ with composition and identities given by: 
$(f_2,\F_2)\circ(f_1,\F_1):=(f_2\circ f_1, \F_1\circ f_1^\bullet(\F_2)\circ \Theta^{\Es_3}_{f_2,f_1})$, 
and $\iota_{(\Es,\pi,\Xs)}:=(\iota_{\Xs},\iota_{\Xs}^\pi)$,  
where $f_2\circ f_1:=({f_2}_\Delta\circ {f_1}_\Delta, {f_2}_\Gs\circ {f_1}_\Gs)$, 
$\Theta_{f_1,f_2}^{\Es_3}:(F_2\circ f_1)^\bullet(\Es_3)\to f_1^\bullet(f_2^\bullet(\Es_3))$ is the canonical isorphism of standard pull-backs and $f_1^\bullet(\F_2):f_1^\bullet (f_2^\bullet(\Es_3))\to f_1^\bullet(\Es_2)$ is the image of the map 
$\F_2:f_2^\bullet(\Es_3)\to\Es_2$ under the standard pull-back functor $f_1^\bullet$ between the categories of bundles over $\Xs_2$ and $\Xs_1$. 
\end{remark}

In our previous work (see~\cite{BCL4} and also~\cite[Section~4.2]{BCL0}) we obtained the following horizontal categorification of Gel'fand duality theorem for commutative full C*-categories:
\begin{theorem}[Categorified Gel'fand Theorem] \label{th: catgel}
There is a canonical duality $(\Gamma_\Af,\Sigma_\Af)$ between the category $\Tf_\Af$ of  
morphisms between spaceoids and the category $\Af$ of object-bijective $*$-functors between small commutative full C*-categories.  
\end{theorem} 

\section{Monoidal Categories and Enriched Bundles.}\label{sec: monoidal}

For convenience of the reader, we recall the standard definition of (weak-)monoidal category (see for example 
T.~Leinster~\cite[Section~1.2]{Le} and M.~Kelly~\cite[Section~1.1]{Ke}). 
\begin{definition}\label{def: monoidal}
A \emph{monoidal category} is a category $\Mf$ equipped with a covariant bifunctor $\otimes:\Mf\times\Mf$ and an identity object 
$I\in\Ob_\Mf$ such that for all objects $A,B,C\in \Ob_\Mf$ there are natural isomorphisms 
\begin{gather*}
(A\otimes B)\otimes C \xrightarrow{\alpha_{A,B,C}} A\otimes (B\otimes C), \\
A\otimes I \xrightarrow{\rho_A} A, \quad I\otimes B\xrightarrow{\lambda_B} B  
\end{gather*}
making any diagram with objects in $\Mf$ and whose morphisms are obtained by (eventually repeated) applications of the functor 
$\otimes$ to instances of $\alpha$, $\lambda$, $\rho$, $\iota$ and their inverses commutative.
\end{definition}

In a completely similar way, we propose here a definition of (weak-)monoidal $*$-category:  
\begin{definition}\label{def: *-monoidal}
A \emph{(weak-)monoidal $*$-category} 
$\Mf^*$ is just a monoidal category equipped with a 
contravariant 
functor $\dag: \Mf^*\to\Mf^*$ such that for all $A,B\in \Ob_{\Mf^*}$ there are natural isomorphisms
\begin{gather*}
(A^\dag)^\dag  \xrightarrow{\beta_A} A, \quad 
(A\otimes B)^\dag \xrightarrow{\gamma_{A,B}} 
B^\dag\otimes A^\dag,  
\end{gather*}
making any possible diagram in the category $\Mf^*$ whose morphisms are obtained by (eventually reapeated) applications of the functors $\otimes$ and $\dag$ to instances of $\alpha$, $\beta$, $\gamma$, $\lambda$, $\rho$, $\iota$ and their inverses commutative.

For an object $A$ (or respectively an arrow $f:A\to B$) in $\Mf^*$, we will say that $A^\dag$ (respectively $f^\dag:B^\dag\to A^\dag$) is its \emph{monoidal dual}.  
\end{definition}

\begin{remark} 
As a basic example illustrating the previous definition, the reader can consider the category of continuous linear maps between Hilbert spaces. In this case the monoidal dual of an object is just the ordinary dual and the monoidal dual of an arrow is the ordinary transposed map. The map $\beta_H^{-1}:H\to (H^\dag)^\dag$ is the canonical isomorphism of a Hilbert space $H$ with its double dual $(H^\dag)^\dag$ associating to every $x\in H$ the evaluation in $x$ of elements of the dual (every Hilbert space is reflexive). Similarly $\gamma_{H,K}^{-1}:K^\dag\otimes H^\dag\to (H\otimes K)^\dag$ is the canonical isomorphism associating to $\psi\otimes\phi\in K^\dag\otimes H^\dag$ the map $h\otimes k\mapsto \phi(h)\psi(k)$. Other more involved examples are presented in section~\ref{sec: examples}.  

Note that in the definition of (weak-)monoidal $*$-category we have required the functor $\dag$ to be contravariant in order to ``match'' with the contravariance of the involution of a $*$-category. It is of course possible to give a  similar definition for covariant functors and examples can be obtained considering conjugation of linear continuous maps on complex Hilbert spaces. 
\end{remark}

\begin{remark}
Note that from time to time, in the following, we will need to deal with a partially defined monoidal composition bifunctor and hence with several ``partial'' monoidal identities (as for example in the case of tensor product of bimodules) that otherwise satisfies the very same axioms in the definition of (weak-)monoidal category. Such ``many-objects'' version  of a \hbox{(weak-)monoidal} category can be precisely formalized via the notion of J.~B\'enabou's \emph{bicategory} (see for example S.~Lack~\cite[Section~4]{L}) that we will freely use.

In a completely parallel way, a (weak-)monoidal $*$-category admits a ``many-object'' horizontal categorification via the notion of 
\emph{involutive bicategory}.  
\end{remark}

Categories and Fell bundles are special instances of a general concept of enriched bundle that now we proceed to define, following closely M.~Kelly's idea of enriched category~\cite{Ke}.  

\begin{definition} \label{def: en-cat}
Let $\Mf$ be a monoidal category. An \emph{$\Mf$-enriched categorical bundle} over the base category $\Xs$ is a bundle\footnote{For now $\pi$ can be assumed to be just a surjective map. Continuity of $\pi$ will be required whenever $\Es$ and $\Xs$ are topological spaces. Whenever the objects of the monoidal category $\Mf$ are Banach spaces (or more generally, disjoint union of Banach spaces), we require $(\Es,\pi,\Xs)$ to be a Banach bundle. In this last case, also the maps $\mu_{x,y}\circ \otimes:\Es_x\times\Es_y\to\Es_{x\circ y}$ are supposed to be bilinear such that $\|\mu_{x,y}(e,f)\|\leq\|e\|\cdot\|f\|$. } $\pi:\Es\to\Xs$ such that:
\begin{itemize}
\item 
for all $x\in \Xs$, the fiber $\Es_x:=\pi^{-1}(x)$ is an object of the monoidal category $\Mf$; 
\item 
for all pairs $(x,y)\in\Xs^2$ of composable arrows in $\Xs$, there is an associated morphism $\mu_{x,y}\in \Hom_\Mf(\Es_x\otimes \Es_y,\Es_{x\circ y})$ that makes the following diagram commutative:
\begin{equation*}
\xymatrix{
\Es_x\otimes(\Es_y\otimes \Es_z) \ar[d]_{\iota_{\Es_x}\otimes\mu_{y,z}} \ar[rr]^{\alpha_{\Es_x,\Es_y,\Es_z}} &   
& (\Es_x\otimes \Es_y)\otimes \Es_z \ar[d]^{\mu_{x,y}\otimes\iota_{\Es_z}} \\
\Es_x\otimes \Es_{y\circ z} \ar[dr]_{\mu_{x,y\circ z}} & &
\Es_{x\circ y}\otimes \Es_z \ar[dl]^{\mu_{x\circ y, z}} \\
& \Es_{x\circ y\circ z}. 
}
\end{equation*} 
\item
for all identities $y\in \Xs^o$, there is a morphism $j_y\in\Hom_\Mf(I,\Es_y)$ that makes the following diagrams commutative:
\begin{equation*}
\xymatrix{
\Es_x\otimes I\ar[dr]_{\rho_{\Es_x}} \ar[rr]^{\iota_{\Es_x}\otimes j_y} & &  \Es_x\otimes \Es_y, \ar[dl]^{\mu_{x,y}}\\
& \Es_x & 
}
\quad 
\xymatrix{
I \otimes \Es_z\ar[dr]_{\lambda_{\Es_z}} \ar[rr]^{j_y\otimes \iota_{\Es_z}} & & \Es_y\otimes \Es_z. \ar[dl]^{\mu_{y,z}}\\
& \Es_z & 
}
\end{equation*}
\end{itemize}
\end{definition}

\begin{remark} 
A basic inspiring example, that the reader might keep in mind in order to justify the following abstract definitions, can be obtained considering 
the Morita-Rieffel $*$-category $\Xs$ (consisting of isomorphism classes of unital Hilbert C*-modules over unital commutative \hbox{C*-algebras}) and choosing for any equivalence class $x\in \Xs$ a given representative Hilbert \hbox{C*-module} $\Es_x\in x$. 
In this case, the (weak-)monoidal $*$-category $\Mf^*$, mentioned in definition~\ref{def: en-fell} here below, is the category of unital Hilbert 
C*-modules with monoidal product the usual tensor product of Hilbert C*-modules and monoidal duals given by the Rieffel duals~\cite{R}. 
\end{remark}

\begin{definition}\label{def: en-fell}
Let $\Mf^*$ be a (weak-)monoidal $*$-category. Let $\Xs$ be a $*$-category. 

An \emph{$\Mf^*$-enriched $*$-categorical bundle}\footnote{Again, whenever the object of $\Mf^*$ are Banach spaces (or more generally disjoint union of Banach spaces) we ask $\nu_x:\Es_x^\dag\to\Es_{x^*}$ to be linear and isometric $\|\nu_x(e)\|=\|e\|$.} over the base $*$-category $\Xs$ is a categorical bundle enriched in the 
(weak-)monoidal $*$-category $\Mf^*$ such that 
\begin{itemize}
\item
for all $x\in \Xs$, there is an associated morphism 
$\nu_x\in \Hom_{\Mf^*}((\Es_x)^\dag,\Es_{x^*})$ 
such that the following diagrams commute: 
\begin{equation*}
\xymatrix{
((\Es_x)^\dag)^\dag \ar[rr]^{\beta_{\Es_x}} \ar[rd]_{(\nu_x)^\dag} & & \Es_x 
\\
& (\Es_{x^*})^\dag \ar[ur]_{\nu_{x^*}} & 
}
\xymatrix{
(\Es_x\otimes \Es_y)^\dag \ar[d]_{(\mu_{x,y})^\dag}  \ar[rr]^{\gamma_{\Es_x,\Es_y}} & & (\Es_y)^\dag\otimes(\Es_x)^\dag  
\ar[d]^{\nu_y\otimes\nu_x}\\
(\Es_{xy})^\dag \ar[dr]_{\nu_{xy}} & & \Es_{y^*}\otimes\Es_{x^*} \ar[dl]^{\mu_{y^*,x^*}} \\
& \Es_{(xy)^*}=\Es_{y^*x^*} &    
}
\end{equation*}
\end{itemize}
\end{definition}

\begin{remark}
Categorical bundles enriched in a bicategory and $*$-categorical bundles enriched in an involutive bicategory can be defined in exactly the same way, as long as the assignement $x\mapsto \Es_x$ is weakly $*$-functorial modulo the morphisms $\mu,j,\nu$.
It is understood that, if necessary, in all the subsequent material, whenever the term monoidal ($*$-)category appears, it can be substituted by (involutive) bicategory. 

Clearly, for a (concrete) monoidal category $\Mf$, 
an $\Mf$-enriched category (as defined in T.~Leinster~\cite[Section~1.3]{Le} for example) is just an $\Mf$-enriched categorical bundle whose base category $\Xs$ is a discrete equivalence relation. 

Clearly Fell bundles are just $\Mf^*$-enriched $*$-categorical bundles where $\Mf^*$ is the  involutive bicategory of Hilbert \cs-bimodules. 

A spaceoid can be seen as an $\Mf^*$-enriched $*$-categorical bundle over the equivalence relation $\Xs:=\Rs_\Os$, where the (weak-)monoidal $*$-category $\Mf^*$ is the category of Hermitian line bundles over the compact Hausdorff topological space $X$. 

Note also that since $\Mf$ and respectively $\Mf^*$ are concrete categories (i.e.~are subategories of the category of sets) it is always possible to use the canonical isomorphisms $\alpha, \rho, \lambda$ in definition~\ref{def: monoidal} (together with the isomorphims $\beta,\gamma$ in 
definition~\ref{def: *-monoidal} for the $*$-monoidal situation) to introduce on the total space $\Es$ of the bundle $(\Es,\pi,\Xs)$ a structure of category (respectively of $*$-category) in such a way that $\pi:\Es\to\Xs$ becomes a suitable functor (respecively a $*$-functor). 
\end{remark}

\begin{definition}
A $*$-categorical bundle over an inverse $*$-category $\Xs$ enriched in a concrete (weak-)monoidal $*$-category $\Mf^*$ whose objects are equipped with a Banach norm is called a \emph{Fell bundle enriched in the (weak-)monoidal $*$-category $\Mf^*$} if it satisfies the following version of the C*-property $\|\mu_{x^*,x}(\nu_x(e),e)\|=\|e\|^2$, for all $e\in \Es_x$ and, for all $e\in \Es_x$, there exists $h\in \Es_{x^*x}$ such that $\mu_{x^*,x}(\nu_x(e),e)=\mu_{x^*x,x^*x}(\nu_{x^*x}(h),h)$. 
\end{definition}

\begin{remark}
Given a categorical bundle $(\Es,\pi,\Xs)$, enriched in the monoidal category $\Mf$ and given a functor 
$f:\Zs\to\Xs_2$, the standard $f$-pull-back $(f^\bullet(\Es),\pi^f,\Zs)$ of $(\Es,\pi,\Xs)$ has a natural structure of categorical bundle enriched in the monoidal category $\Mf$. 
Similarly the $f$-pull-back of the involutive categorical bundle $(\Es,\pi,\Xs)$ enriched in the \hbox{(weak-)monoidal} \hbox{$*$-category} $\Mf^*$ under a 
$*$-functor $f:\Zs\to\Xs$ has a natural structure of involutive categorical bundle enriched in the same (weak-)monoidal $*$-category $\Mf^*$. 
\end{remark}

\begin{definition}\label{def: mor-X}
A \emph{morphism $\F:(\Es^1,\pi^1,\Xs)\to(\Es^2,\pi^2,\Xs)$ of categorical bundles}, over the same category $\Xs$, \emph{enriched in the monoidal category $\Mf$}, is a map $\F:\Es^1\to\Es^2$ satisfying $\pi^2\circ\F=\pi^1$ such that for all $x\in \Xs$, 
$\F_x:\Es^1_x\to\Es^2_x$ is a morphism in $\Mf$ with the properties 
\begin{gather*}
\F_x\otimes(\F_y\otimes\F_z)\circ \alpha_{\Es^1_x,\Es^1_y,\Es^1_z}=\alpha_{\Es^2_x,\Es^2_y,\Es^2_z}\circ (\F_x\otimes\F_y)\otimes\F_z, \quad \forall (x,y,z)\in \Xs^{(3)}, 
\\
\F_x\circ\rho_{\Es^1_x}=\rho_{\Es^2_x}\circ(\F_x\otimes\iota_I),  \quad 
\F_x\circ\lambda_{\Es^1_x}=\lambda_{\Es^2_x}\circ (\iota_I\otimes \F_x), \quad \forall x\in \Xs, 
\\ 
\F_{x\circ y}\circ \mu^1_{x,y}=\mu^2_{x,y}\circ (\F_x\otimes\F_y), \quad \forall (x,y)\in \Xs^{(2)}, 
\\
j^2_y\circ\iota_I=\F_y\circ j^1_y, \quad \text{for all the identities $y$ in $\Xs^o$}.
\end{gather*}

Whenever $\Es^1,\Es^2$ are topological spaces, we require $\F$ to be a continuous map and whenever the object of the monoidal category $\Mf$ are (disjoint union of) Banach spaces, we require $\F$ to be fiberwise linear. 

Note again that for a categorical bundle $(\Es,\pi,\Xs)$ enriched in the monoidal category $\Mf$, the total space $\Es$ can be equipped with the structure of a category and a morphism $\F$ as defined above becomes immediately a functor. 
\end{definition}

\begin{definition}
Given two categorical bundles enriched in a monoidal category $(\Es_j,\pi_j,\Xs_j)$, for $j=1,2$, a morphism between them is a pair 
$(f,\F):(\Es_1,\pi_1,\Xs_1)\to(\Es_2,\pi_2,\Xs_2)$ such that $f:\Xs_1\to\Xs_2$ is a functor and 
$\Fs: (f^\bullet(\Es_2),\pi_2^f,\Xs_1)\to(\Es_1,\pi_1,\Xs_1)$ is a morphisms of enriched categorical bundles, over the category $\Xs_1$ as specified in definition~\ref{def: mor-X}. 
\end{definition}

In a completely similar way, for Fell bundles enriched in a (weak-)monoidal $*$-category, 
\begin{definition} \label{def: mor-X*}
A \emph{morphism $\F:(\Es^1,\pi^1,\Xs)\to(\Es^2,\pi^2,\Xs)$ of $*$-categorical bundles}, on the same $*$-category $\Xs$, \emph{enriched in the same (weak-)monoidal $*$-category $\Mf^*$}, is a morphism of categorical bundles over $\Xs$ such that the following further properties are satisfied: 
\begin{gather*}
\F_x\circ\beta_{\Es^1_x}=\beta_{\Es^2_x}\circ (\F_x^\dag)^\dag, \quad \forall x\in \Xs, 
\\
(\F_y^\dag\otimes\F_x^\dag)\circ\gamma_{\Es^1_x,\Es^1_y}=\gamma_{\Es^2_x,\Es^2_y}\circ (F_x\otimes\F_y)^\dag, \quad \forall x,y\in \Xs, 
\\
\F_{x^*}\circ \nu_x=\nu_x\circ \F_x^\dag, \quad \forall x\in \Xs. 
\end{gather*}
Again, since the total space of a $*$-categorical bundle enriched in a (weak-)monoidal $*$-category is naturally equipped with the structure of a 
$*$-category, every such morphism $\F$ becomes a $*$-functor. 
\end{definition}

\begin{definition}
Given two Fell bundles $(\Es_j,\pi_j,\Xs_j)$, $j=1,2$, enriched in a \hbox{(weak-)monoidal} $*$-category,  
a morphism between them is a pair 
$(f,\F):(\Es_1,\pi_1,\Xs_1)\to(\Es_2,\pi_2,\Xs_2)$ such that $f:\Xs_1\to\Xs_2$ is a $*$-functor and 
$\Fs: (f^\bullet(\Es_2),\pi_2^f,\Xs_1)\to(\Es_1,\pi_1,\Xs_1)$ is a morphisms of \hbox{$*$-categorical} bundles over the $*$-category $\Xs_1$, as specified in definition~\ref{def: mor-X*}. 
\end{definition}

The relations between categories, categorical bundles and their $\Mf$-enriched versions can be formalized in the following diagram 
(the involutive case ($*$-) being optional):
\begin{equation*}
\xymatrix{
\text{($*$-)Categories} \ar@{^{(}->}[rr] \ar@{^{(}->}[d]& & \text{($*$-)Enriched Categories} \ar@{^{(}->}[d] 
\\
\text{($*$-)Categorical Bundles} \ar@{^{(}->}[rr] & & \text{($*$-)Enriched Categorical Bundles} 
}
\end{equation*}

The ``embeddings'' from ``left'' to ``right'' are just examples of \emph{enrichment functors}, while the ``embeddings'' from top to bottom are examples of what we might call \emph{bundlefication functors}.

\section{Spectral Spaceoids for Saturated Unital Fell Bundles}\label{sec: equivalence}

We now come to our proposed goal: to give alternative descriptions of the spectrum $\Sigma(\Cs)$ of a commutative full C*-category $\Cs$  
in terms of Fell bundles enriched in a suitable \hbox{(weak-)monoidal} $*$-category. 

\subsection{Relevant Examples of (Weak-)Monoidal $*$-Categories}\label{sec: examples}

We introduce here in some detail a few of the monoidal categories and (weak-)monoidal \hbox{$*$-categories} that are relevant in our discussion of spectral spaceoids for commutative full \hbox{C*-categories} and saturated commutative Fell bundles. 

\begin{example} 
Let $\As$ be a fixed unital C*-algebra and let $\Mf_\As$ denote the category with objects unital Hilbert C*-bimodules over the 
C*-algebra $\As$, with morphisms given by right and left $\As$-linear maps subjected to the usual composition of maps. 
The category $\Mf_\As$ is a \hbox{(weak-)monoidal} $*$-category with monoidal product given by the Rieffel tensor product of 
Hilbert C*-bimodules (see M.~Rieffel~\cite{R}), monoidal identity given by the C*-algebra $\As$ (considered as a Hilbert \hbox{C*-bimodule} over itself) and monoidal dual given by the Rieffel dual of a Hilbert C*-bimodule~\cite{R}. 

Whenever the C*-algebra $\As$ is commutative, the full subcategory $\Sf\Mf_\As$ of symmetric Hilbert C*-bimodules over $\As$, is again a 
(weak-)monoidal $*$-category.
\end{example}

\begin{example}
Let $\Xs$ be a fixed compact Hausdorff topological space and let $\Hf_\Xs$ be the category with objects Hilbert bundles $(\Es,\pi,\Xs)$ over the space $\Xs$ and morphisms given by the ``bundle maps'' i.e.~the functions $F:\Es_1\to\Es_2$ such that $\pi_2\circ F=\pi_1$ subjected to the usual composition of maps. 
The category $\Hf_\Xs$ is a (weak-)monoidal $*$-category with monoidal product given by the fiberwise tensor product of the Hilbert bundles, identity provided by the trivial line bundle and monoidal dual given by the fiberwise conjugate Hilbert bundle.  

The (weak-)monoidal $*$-categories $\Hf_\Xs$ 
and $\Sf\Mf_{C(\Xs)}$ 
are equivalent 
as a consequence of the more general Takahashi duality~\cite{Ta1,Ta2}. Whenever we restrict to the full subcategory $\Vf_\Xs$ of finite-rank Hilbert bundles\footnote{These are necessarily Hermitian vector bundles (see J.~Fell-R.~Doran~\cite[Section~13]{FD}).} we recover the Hermitian version of Serre-Swan equivalence with the full subcategory $\Pf_{C(\Xs)}$ of finite projective unital symmetric Hilbert 
C*-bimodules over $C(\Xs)$. 
\end{example}

\begin{example}
More generally we can consider the category $\Mf_\bullet$ with objects unital Hilbert \hbox{C*-bimodules} over (possibly different pairs of) unital 
C*-algebras with morphisms given by the triples $(\phi,\Phi,\psi): {}_\As\Ms_\Bs\to {}_{\As_1}\Ns_{\Bs_1}$ with 
$\phi:\As\to\As_1$ and $\psi:\Bs\to\Bs_1$  unital \hbox{$*$-homomorphisms} of C*-algebras and $\Phi:\Ms\to\Ns$ continous map such that 
$\Phi(a\cdot x\cdot b)=\phi(a)\cdot\Phi(x)\cdot \psi(b)$, for all $x\in \Ms$, $a\in \As$ and $b\in\Bs$. 
The category $\Mf_\bullet$ is a bicategory i.e.~a ``many-object (weak-)monoidal $*$-category'' where the partial ``monoidal'' product is given by Rieffel tensor product of Hilbert \hbox{C*-bimodules}, the ``monoidal'' identities are given by the unital C*-algebras (considered as Hilbert C*-bimodules over themselves) and the ``monoidal'' involution is given by the Rieffel dual of a Hilbert C*-bimodule. 
Whenever the algebras are commutative, the full subcategory $\Sf\Mf_\bullet$ of symmetric Hilbert C*-bimodules is again a multi-object 
(weak-)monoidal $*$-category. 
\end{example}

\begin{example} \label{ex: h}
Let $\Hf_\bullet$ denote the category with objects Hilbert bundles over compact Hausdorff topological spaces and morphisms given by pairs $(f,\F):(\Es_1\pi_1,\Xs_1)\to(\Es_2,\pi_2,\Xs_2)$, where $f:\Xs_1\to\Xs_2$ is continuous map and 
$\F:(f^\bullet(\Es_2),\pi_2^f,\Xs_1)\to(\Es_1,\pi_1,\Xs_1)$ a bundle map in the category $\Hf_{\Xs_1}$ defined on the standard 
$f$-pull-back of $(\Es_2,\pi_2,\Xs_2)$. 
The category $\Hf_\bullet$ is again a bicategory with partial ``monoidal'' product given by fiberwise tensor product of Hilbert bundles, ``monoidal'' identities given by trivial line bundles and ``monoidal'' involution given by the fiberwise conjugate of a Hilbert bundle.

The category $\Hf_\bullet$ is in duality with the category $\Sf\Mf_\bullet$ via Takahashi theorem~\cite{Ta1,Ta2}.

\end{example} 

The previous examples of monoidal $*$-categories and $*$-bicategories can be further generalized considering bimodules for C*-categories or even more generally bimodules for Fell bundles. 

\begin{example}
Following P.~Mitchener~\cite[Section~8]{M1}, we consider the category $\Mf_{\text{C*-cat-}\Cs}$ with objects Hilbert C*-bimodules over a C*-category $\Cs$ and continuous $\Cs$-linear maps as morphisms. 

The category $\Mf_{\text{C*-cat-}\Cs}$ is a (weak-)monoidal $*$-category with monoidal product given by the tensor product of 
$\Cs$-bimodules and monoidal involution given by the dual of such $\Cs$-bimodules.

The example can be further extended considering the category $\Mf_{\text{Fell-}\Es}$ of Hilbert \hbox{C*-bimodules} over a given Fell bundle $\Es$. 
\end{example}

\subsection{Equivalent Notions of Spaceoid}

Consider the category of isomorphism of transitive groupoids\footnote{These are groupoids $\Gs$ such that for all $A,B\in \Ob_\Gs$ the set $\Gs_{AB}\neq\varnothing$.} and denote by $\Rf$ its full subcategory of total equivalence relations. By $\Kf$ we denote the category of continuous maps of compact Hausdorff spaces. 
Furthermore $\Lf_\bullet$ denotes the full subcategory of the category $\Hf_\bullet$ in example~\ref{ex: h} determined by Hermitian line bundles over compact Hausdorff spaces and by $\Lf_\Kf$ the disjoint union of of the categories $\Lf_X$ for $X\in \Kf$. 

\begin{theorem}
There is a natural isomorphism between the category $\Tf_\Af$ of spaceoids for small commutative full C*-categories and the category $\Ff^{\Lf_\Kf}_\Rf$ of Fell bundles over total equivalence relations enriched in the (weak-)monoidal $*$-category $\Lf_\Kf$ of Hermitian line bundles over a given compact Hausdorff space. 
\end{theorem}
\begin{proof} 
Given a spaceoid $(\Es,\pi,\Delta_X\times\Rs_\Os)$ in the category $\Tf_\Af$, for all $AB\in \Rs_\Os$, its restriction  
to the base-block $\Delta_X\times \{AB\}$, denoted for short by $\Es_{AB}$, is a Hermitian line bundle over the space 
$\Delta_X\times\{AB\}$ that is naturally isomorphic to a Hermitian line bundle on the compact Hausdorff space $X$. 

For every pair of composable arrows $AB$, $BC\in \Rs_\Os$ there is a $\Ff_\Rf^{\Lf_\Kf}$ linear multiplication map 
\hbox{$\mu_{AB,BC}:\Es_{AB}\otimes\Es_{BC}\to\Es_{AC}$}, obtained by restriction of the multiplication on the total space of the spaceoid, that clearly satisfies the pentagon diagram in definition~\ref{def: en-cat}. 

For every $AA\in \Rs_\Os$, the restriction $\Es_{AA}$ to the diagonal block is naturally isomorphic to the trivial line bundle over $X$ and these isomorphisms with the trivial $\CC$-line bundle over $X$ satisfy the triangle diagrams in definition~\ref{def: en-cat}. 

Finally for every $AB\in \Rs_\Os$, there is a (fiberwise linear) morphism $\nu_{AB}:\Es_{AB}^\dag\to\Es_{BA}$, obtained by restriction of the involution map on the total space of the spaceoid, that satisfies the diagrams in definition~\ref{def: en-fell}.

For every morphism $(f,\F):(\Es_1,\pi_1,\Delta_{X_1}\times\Rs_{\Os_1})\to(\Es_2,\pi_2,\Delta_{X_2}\times\Rs_{\Os_2})$ of spaceoids, 
$f_\Rs$ is already an isomorphism of total equivalence relations and so a morphism in $\Rf$; the pair $\H_{AB}:=(f_\Delta,\F|_{AB})$, for all $AB\in \Rs_{\Os_1}$, is a morphism in $\Lf_\bullet$ from $(\Es_2)_{f_\Rs(AB)}$ to $\Es_{AB}$ that satisfies all the requirements in definitions~\ref{def: mor-X} and~\ref{def: mor-X*} and hence the pair $(f_\Rs,\H)$ is a morphism in the category 
$\Ff_\Rf^{\Lf_\Kf}$. Is it possible to check that the map $\Tg$ just described from objects and morphisms of $\Es_\As$ to object and morphisms of $\Ff_\Rf^{\Lf_\Kf}$ is functorial. 

\smallskip

In the reverse direction, given a Fell bundle $(\Fs,\rho,\Rs_\Os)$ over a total equivalence relation $\Rs_\Os$ enriched in the (weak-)monoidal 
$*$-category $\Lf_X$, for every $AB\in \Rs_\Os$, the fiber $\Fs_{AB}$ is a Hermitian line bundle over the compact Hausdorff space $X$ and, as already noted, the total space given by the disjoint union $\Es:=\cup_{AB\in \Rs_\Os}\Fs_{AB}$ becomes a 
$*$-category, the projection $\pi$ onto the product topological $*$-category$ \Xs:=\Delta_X\times\Rs_\Os$ defined by glueing the projections $\pi_{AB}$ of the Hermitian bundles $\Fs_{AB}$ is a $*$-functor. Clearly $(\Es,\pi,\Xs)$ is a Banach bundle (since it is a union of Hermitian vector bundles) and hence a spaceoid in $\Tf_\Af$. 

Given now a morphism $(f,\F):(\Fs_1,\rho_1,\Rs_{\Os_1})\to(\Fs_2,\rho_2,\Rs_{\Os_2})$ in $\Ff_\Rf^{\Lf_\Kf}$, where for every 
\hbox{$AB\in \Rs_{\Os_1}$}, $\F_{AB}:=(\phi,\Phi_{AB})$ with fixed $\phi:X\to Y$ is a morphism of Hermitian line bundles, we define a new morphism $(h,\H)$ of the corresponding spaceoids in $\Tf_\Af$ by taking $h_\Rs:=f$, $h_\Delta:=\phi$ and $\H|_{AB}:=\Phi_{AB}$. 
Again the map $\Sg$ here described from objects and morhisms of $\Ff_\Rf^{\Lf_\Kf}$ to $\Tf_\Af$ is functorial and it is an inverse of the functor $\Fg:\Tf_\Af\to\Ff_\Rf^{\Lf_\Kf}$ defined above. 
\end{proof}

\begin{theorem} 
There is a natural isomorphism between the category $\Es_\Af$ of spaceoids for small commutative full C*-categories and the category $\Ff_\Kf^{\Mf_{\text{C*-cat-}\Of}}$ of Fell bundles over compact Hausdorff topological spaces enriched in the bicategory 
$\Mf_{\text{C*-cat-}\Of}$ of small one-dimensional \hbox{C*-categories} with a given set of objects. 
\end{theorem}
\begin{proof}
Given a spaceoid $(\Es,\pi,\Delta_X\times\Rs_\Os)$, we know that for every $p\in\Delta_X$, $\Es_p$ is a C*-category over the equivalence relation $\Rs_\Os$ and of course each $\Es_p$ is a ``monoidal'' identity in $\Mf_{\text{C*-cat-}\Of}$. Since there are no non-trivial composable arrows on the base space $\Delta_X$, the requirements on the canonical isomorphism $\mu,j,\nu$ are clearly satisfied. The topology of the total space $\Es$ assures that the bundle of C*-categories $p\mapsto \Es_p$ is a trivial Fell bundle enriched in $\Mf_{\text{C*-cat-}\Of}$.
For every morphism of spaceoids $(f,\F)$, the pair $(f_\Delta, \H)$, where $\H$ is the family of maps $p\mapsto \H_p$ with 
$\H_p:=\F|_{\Es_p}$, is a morphisms in $\Mf_{\text{C*-cat-}\Of}$. The map $\Tg':\Tf_\Af\to\Mf_{\text{C*-cat-}\Of}$ that we have just described is functorial. 

\medskip 

In the revese direction, given a Fell bundle $(\Fs,\rho,\Delta_X)$ enriched in $\Mf_{\text{C*-cat-}\Of}$, we have that $\Fs_p$ is a monoidal identity in $\Mf_{\text{C*-cat-}\Of}$, for all $p\in \Delta_X$, and hence $\Fs_p$ is a one-dimensional \hbox{C*-category} over the equivalence relation $\Rs_\Os$, with family of objects $\Os$ fixed. The disjoint union of such \hbox{C*-categories} defines a bundle over 
$\Delta_X\times\Rs_\Os$ that is a categorical $*$-bundle. From the requirements on the topology of $\Fs$ we have that the space 
$\Es:=\Fs$ is a Banach bundle over $\Delta_X\times\Rs_\Os$ and hence a spaceoid. 
For every morphism $(h,\H)$ in $\Ff^{\Mf_{\text{C*-cat-}\Of}}_\Kf$, we define $f_\Delta:=h$ that is already a morphism in $\Kf$. 
Furthermore since $\H=(\psi,\Psi)$, with $\psi:\Rs_{\Os_1}\to\psi_{\Os_2}$ a fixed isomorphism of equivalence relations, is such that 
$(\psi,\Psi_p)$, for all $p\in \Delta_{X_1}$, is a morphisms in $\Mf_{\text{C*-cat-}\Of}$, we take $f_\Rs:=\psi$ and we define $\Fs$ to be the disjoint union of all such $\Psi_p$. The pair $(f,\F)$ with $f:=(f_\Delta,f_\Rs)$ is a morphism in the category $\Tf_\Af$. 
The map $\Sg:\Ff^{\Mf_{\text{C*-cat-}\Of}}_\Kf\to\Tf_\Af$ that we just described is funtorial and is an inverse of the functor 
$\Fg:\Ff^{\Mf_{\text{C*-cat-}\Of}}_\Kf\to\Tf_\Af$. 
\end{proof}

\section{Outlook}\label{sec: outlook}

We introduced a variation of the notion of weak monoidal $*$-category, we used it in order to define enriched Fell bundles and we further applied it to the formalization of several equivalent descriptions of the spectral data for commutative full small 
C*-categories. 

It is likely that the notion of spaceoid will undergo some modifications 
in order to encompass spectral theories for more general situations such as commutative saturated Fell bundles over groupoids or more generally inverse involutive categories, but monoidal $*$-categories and enriched Fell bundles will continue to play a role in this more general contexts. 

It will be of interest in the future to consider the vertical categorification extension of these notions in order to give alternative descriptions of the ``higher spaceoids'' that we are using to describe the spectra of commutative full small 
$n$-C*-categories~\cite{BCLS} (see also the slides~\cite{B2} for some more details on this work in progress). 
We plan to return to this topic with more details elsewhere. 

\emph{Acknowledgments.} 

\smallskip
 
{\small
P.~Bertozzini acknowledges the kind support in Kyoto provided by Professors T.~Natsume, Y.~Maeda and H.~Moriyoshi during the 
RIMS International Conference on Noncommutative Geometry and Physics in November 2010. 
}

\end{document}